\documentclass[reqno]{amsart}
\usepackage{latexsym,amsmath,amssymb,amscd, amsthm}
\usepackage{eucal}
\usepackage[utf8]{inputenc}
\usepackage[all]{xy}
\usepackage{enumerate}
\usepackage{color}

\usepackage{hyperref} 
\makeatletter
\@addtoreset{figure}{section}
\def\thefigure{\thesection.\@arabic\c@figure}
\def\fps@figure{h,t}
\@addtoreset{table}{bsection}

\def\thetable{\thesection.\@arabic\c@table}
\def\fps@table{h, t}
\@addtoreset{equation}{section}

\newtheorem{theorem}{Theorem}
\newtheorem*{theorem*}{Theorem}

\newtheorem{lemma}[theorem]{Lemma}

\newtheorem{proposition}[theorem]{Proposition}
\newtheorem{remark}[theorem]{Remark}

\numberwithin{theorem}{section}
\numberwithin{equation}{section}

\newcommand{\0}{{\bf 0}}
\renewcommand{\1}{{\bf 1}}

\newcommand{\Ad}{{\rm Ad}}
\newcommand{\ad}{{\rm ad}}

\newcommand{\Bun}{\text{{\boldmath{$\mathfrak{B}$}}}}

\newcommand{\de}{{\rm d}}
\newcommand{\ee}{{\rm e}}
\newcommand{\End}{{\rm End}}

\newcommand{\Hom}{{\rm Hom}}
\newcommand{\I}{{\rm I}}

\newcommand{\id}{{\rm id}}
\newcommand{\ie}{{\rm i}}
\newcommand{\Ker}{{\rm Ker}\,}

\newcommand{\Prim}{{\rm Prim}}
\newcommand{\pr}{{\rm pr}}

\renewcommand{\Re}{{\rm Re}}

\newcommand{\rk}{{\rm rk}}

\newcommand{\spec}{{\rm spec}}

\newcommand{\CC}{{\mathbb C}}

\newcommand{\KK}{{\mathbb K}}
\newcommand{\NN}{{\mathbb N}}

\newcommand{\RR}{{\mathbb R}}
\newcommand{\TT}{{\mathbb T}}
\newcommand{\ZZ}{{\mathbb Z}}

\newcommand{\Bc}{{\mathcal B}}
\newcommand{\Cc}{{\mathcal C}}

\newcommand{\Gc}{{\mathcal G}}
\newcommand{\Hc}{{\mathcal H}}

\newcommand{\Kc}{{\mathcal K}}

\newcommand{\Oc}{{\mathcal O}}

\newcommand{\Vc}{{\mathcal V}}
\newcommand{\Xc}{{\mathcal X}}

\newcommand{\ag}{{\mathfrak a}}
\newcommand{\cg}{{\mathfrak c}}
\newcommand{\dg}{{\mathfrak d}}
\renewcommand{\gg}{{\mathfrak g}}
\newcommand{\hg}{{\mathfrak h}}
\newcommand{\kg}{{\mathfrak k}}

\renewcommand{\ng}{{\mathfrak n}}

\newcommand{\tg}{{\mathfrak t}}
\newcommand{\zg}{{\mathfrak z}}

\newcommand{\matt}[2]
{\ensuremath{\begin{pmatrix}
			{#1} & 0 \\
			0 & {#2}
\end{pmatrix}}}

\newcommand{\mattt}[3]
{\ensuremath{\begin{pmatrix}
			{#1} & & 0\\
			& {#2} & \\
			0& & {#3}
\end{pmatrix}}}

\usepackage{mathtools}

\newcommand{\mathsout}[1]
{\bgroup\mathchoice
	{\sbox0{$\displaystyle{#1}$}%
		\usebox0\hspace{-\wd0}%
		\rule[0.5\ht0-0.5\dp0-.5pt]{\wd0}{1pt}}%
	{\sbox0{$\textstyle{#1}$}%
		\usebox0\hspace{-\wd0}%
		\rule[0.5\ht0-0.5\dp0-.5pt]{\wd0}{1pt}}%
	{\sbox0{$\scriptstyle{#1}$}%
		\usebox0\hspace{-\wd0}%
		\rule[0.5\ht0-0.5\dp0-.5pt]{\wd0}{1pt}}%
	{\sbox0{$\scriptscriptstyle{#1}$}%
		\usebox0\hspace{-\wd0}%
		\rule[0.5\ht0-0.5\dp0-.5pt]{\wd0}{1pt}}%
	\egroup}

\title
[Rigidity of the Heisenberg group]
{$C^*$-rigidity of the Heisenberg group}

\author{Ingrid Belti\c t\u a}
\author{Daniel Belti\c t\u a}
\address{Institute of Mathematics ``Simion Stoilow'' of the Romanian Academy,
	P.O. Box 1-764, Bucharest, Romania}

\email{Ingrid.Beltita@imar.ro, ingrid.beltita@gmail.com}
\email{Daniel.Beltita@imar.ro, beltita@gmail.com}
\thanks{We acknowledge financial
	support from the Research Grant GAR 2023 (code 114), supported from the Donors’
	Recurrent Fund of the Romanian Academy, managed by the ”PATRIMONIU”
	Foundation.}

\date{\today}

\begin{document}
\parskip4pt


\begin{abstract}
We prove that the Heisenberg groups can be distinguished from the other connected and simply connected Lie groups via their group $C^*$-algebras. 
The main step of the proof is a characterization of the nilpotent Lie groups among the solvable Lie groups solely in terms of topological properties of their coadjoint orbits.
\end{abstract}
	
\maketitle

\section{Introduction}\label{intro}

We study the question of the extent to which a Lie group can be distinguished from the other Lie groups in terms of its corresponding unitary representation theory as encoded in the group $C^*$-algebra. 
More precisely, 
given a class of locally compact groups $\Gc$, a group $G_0\in \Gc$ is called $C^*$-rigid 
within~$\Gc$ if for arbitrary  $G\in \Gc$ we have 
\begin{equation}
\label{intro_eq1}	
C^*(G_0)\simeq C^*(G) \implies
G_0\simeq G.
\end{equation}
Rigidity problems for  locally compact groups have been considered before with $C^*$-algebras replaced by smaller $L^1$ Banach algebras, or commutative  Banach algebras that inherit more information on the structure of the group, like the Fourier algebras or the Fourier-Stieltjes algebras, or
 unitary representations replaced by representations on Banach spaces with relatively few operators; see \cite{GT22} and the references therein. 
 Some of the simplest examples of non-$C^*$-rigid groups are the tori $\TT^n$ for $n=1,2,\dots$. 
 	This is a sequence of compact abelian Lie groups  
 	whose group $C^*$-algebras are mutually $*$-isomorphic since they are commutative and have homeomorphic spectra.
The situation is considerably more complicated in the case of $C^*$-algebras on non-commutative groups, as there are examples of simply connected Lie groups that are not $C^*$-rigid, 
a phenomenon first pointed out in \cite[Subsect.~4]{Ros76} and later systematically studied, e.g., 
in~\cite[Thms. 2.2 and 2.7]{BB-IEOT} for Lie groups of the form $\RR^m\rtimes \RR$, 
having a 1-codimensional abelian normal closed subgroup.

However, one expects that \textit{a priori} information on the structure of the group $G_0$ would lead to more information on their rigidity properties. 
In the case of abelian groups 
 the rigidity problem is solved, via the Pontryagin duality, within the class of locally  groups. 
 For a discussion on various dual objects of locally compact groups, see \cite{BkH20}.

The study of $C^*$-rigidity  is particularly interesting within the
 class of Lie groups, 
due to their role in Mathematical Physics and Harmonic Analysis.
We will focus on the class
$\mathcal{LG}$ consisting of the  connected, simply connected Lie groups since
these Lie groups are uniquely determined by their Lie algebras.
The rigidity problem is not easier in this class, as the example of exponential Lie groups shows. 
By definition, a Lie group $G$ with its Lie algebra $\gg$ is exponential 
if the exponential map $\exp_G\colon\gg\to G$ is a diffeomorphism; we say that $G$  belongs to $\mathcal{EG}$. 
For groups in the class $\mathcal{EG}$, the Kirillov theory gives a  suitable homeomorphism between the unitary dual space of the group and the the space of coadjoint orbits, which is
 a topological space that is easier to handle and reflects structure of the group. 
However, even  
within $\mathcal{EG}$, not every group is $C^*$-rigid, 
as pointed out above. 
Nevertheless, for the smaller class of connected and simply connected nilpotent Lie groups, there are examples of groups, including the Heisenberg groups, that are $C^*$-rigid within the class $\mathcal{EG}$, see \cite{BB-JTA}.

In the present paper we prove that the Heisenberg groups $H_{n}$, $n\ge 1$, are $C^*$-rigid in the wider class~$\mathcal{LG}$; 
we thus obtain the first examples of noncommutative groups that are $C^*$-rigid within the class $\mathcal{LG}$ (Theorem~\ref{main-main_thm}). 
 To do this, we need the intermediate class of connected and simply connected solvable Lie  groups $\mathcal{SG}$, 
$$
\mathcal{EG}
\subsetneqq\mathcal{SG}\subsetneqq\mathcal{LG}.$$  
Our step-by-step approach to proving \eqref{intro_eq1} 
for $H_{n}\in \mathcal{EG}$ and arbitrary $G\in \mathcal{LG}$ 
could be thus summarized:  if $C^*(G) \simeq C^*(H_n)$, then 
$$
G\in\mathcal{LG}\mathop{\implies}\limits^{\rm Step 1}
G\in\mathcal{SG}\mathop{\implies}\limits^{\rm Step 2}
G\in\mathcal{EG} \mathop{\implies}\limits^{\rm Step 3}
G\simeq H_n .$$    
Here Step 1 is taken care of by Proposition~\ref{solv}, 
while Step 3 relies on the results of  \cite{BB-JTA}. 
The main novelty is required by Step~2, and it may hold an independent interest; 
namely, we obtain the following characterization of nilpotent Lie groups within the class of $\mathcal{SG}$ solely in terms of the topology of their coadjoint orbits.
\begin{theorem}\label{IRS}
Let $G$ be a 1-connected solvable Lie group.
Then $G$ is nilpotent if and only if 
its coadjoint orbits are 
closed subsets of $\gg^*$ and 
simply connected. 
\end{theorem}

Thus the main $C^*$-rigidity result of our paper is the following theorem.

\begin{theorem}\label{main-main_thm}
Let $G$ be a  1-connected Lie group such that 
 $C^*(G) \simeq C^*(H_{n})$ for some $n\ge 1$.
 Then 
 $G\simeq H_{n}$.
\end{theorem}

The structure of our paper is as follows: 
Section~\ref{splitting} develops some algebraic tools that are needed 
for proving Theorem~\ref{IRS} in Section~\ref{proofIRS}. 
Section~\ref{proofMain} is devoted to proving Theorem~\ref{main-main_thm}. 

Throughout this paper, `1-connected' means connected and simply connected. 
Every 1-connected Lie group is denoted by an upper case Roman letter, 
and its Lie algebra is denoted by the corresponding lower case Gothic letter. 
For any Lie algebra $\gg$ with its linear dual space $\gg^*$ we denote by $\langle\cdot,\cdot\rangle\colon\gg^*\times\gg\to\RR$ the corresponding duality pairing. 
We often denote the group actions simply by juxtaposition, and in particular this is the case for the coadjoint action $G\times\gg^*\to\gg^*$, $(g,\xi)\mapsto g\xi$.

\section{Splitting of centreless metabelian Lie algebras}
\label{splitting}

For every finite-dimensional real Lie algebra $\gg$ with its descending central series  $\gg=\gg^1\supseteq\gg^2\supseteq\cdots$ we denote $\gg^\infty:=\bigcap\limits_{j\ge 1}\gg^j$. 
Then $\gg^\infty$ is a characteristic ideal of~$\gg$. 
We recall that a Cartan subalgebra of $\gg$ is a nilpotent subalgebra 
$\hg\subseteq \gg$ which is not an ideal in any strictly larger subalgebra of $\gg$, that is, $\hg= \{ x\in \gg : [x, \hg]\subseteq \hg \}$.

\begin{lemma}
	\label{L1}
	Let $\gg$ be a real Lie algebra with $n:=\dim\gg<\infty$ 
	and, for an arbitrary Cartan subalgebra $\hg\subseteq\gg$, we define the linear subspace $\gg_\hg:=\sum\limits_{h\in\hg}(\ad_\gg h)^n\gg\subseteq\gg$. 
	If the ideal $\gg^\infty$ is abelian, then the Lie algebra $\gg$ is solvable and we have  $\gg^\infty=\gg_\hg$ and  
	$\gg=\gg^\infty\rtimes\hg$. 
\end{lemma}

\begin{proof}
	We have the direct sum decomposition $\gg=\gg_\hg\dotplus\hg$ and the inclusion  $\gg_\hg\subseteq\gg^\infty$ since $\hg$ is a Cartan subalgebra of $\gg$. 
	Moreover $\gg_\hg=\gg^\infty$ since the ideal $\gg^\infty$ is abelian, cf.,  
	\cite[Ch. V, \S 7, page 117]{Se67}. 
	Finally, the fact that the Lie algebra $\gg$ is solvable follows by the semidirect product decomposition $\gg=\gg^\infty\rtimes\hg$ in which both $\gg^\infty$ and $\hg$ are solvable.
\end{proof}

The following lemma is inspired by the proof of a result on finite groups, 
namely 
\cite[9.2.7]{Ro96}. 

\begin{lemma}
	\label{L2}
	If $\gg$ is a finite-dimensional real Lie algebra with trivial centre 
	and $\gg^\infty$ is abelian, then $\gg^\infty$ is a maximal abelian subalgebra of $\gg$. 
\end{lemma}

\begin{proof}
	Let us denote $\cg:=\{x\in\gg : [x,\gg^\infty]=\{0\}\}$. 
	Since $\gg^\infty$ is an ideal in $\gg$, it follows that $\cg$ is an ideal in $\gg$, too. 
	In order to prove the assertion, it suffices to show that $\cg\subseteq\gg^\infty$. 
	The proof of this inclusion relation proceeds in two steps: 
	
	Step 1: 
	We claim that for every nilpotent subalgebra $\hg\subseteq \gg$ with $\hg+\gg^\infty=\gg$ we have $\hg\cap\cg=\{0\}$. 
	
	Let $\zg$ denote the centre of $\hg$, hence $[\zg\cap\cg,\hg]\subseteq[\zg,\hg]=\{0\}$. 
	On the other hand, the definition of $\cg$ implies $[\cg,\gg^\infty]=\{0\}$, hence   $[\zg\cap\cg,\gg^\infty]=\{0\}$. 
	Consequently $[\zg\cap\cg,\gg]=[\zg\cap\cg,\hg+\gg^\infty]=\{0\}$. 
	Thus $\zg\cap\cg$ is contained in the centre of $\gg$, which is trivial by hypothesis, hence $\zg\cap\cg=\{0\}$ or, equivalently, $\zg\cap(\hg\cap\cg)=\{0\}$. 
	
	However, since $\cg\trianglelefteq\gg$, we have $\hg\cap\cg\trianglelefteq\hg$. 
	Therefore, since $\hg$ is a nilpotent Lie algebra and $\zg\cap(\hg\cap\cg)=\{0\}$, 
	we must have $\hg\cap\cg=\{0\}$. 
	(See \cite[Cor. 1.1.15]{CG90}.) 
	
	Step 2: 
	We claim that there exists a nilpotent subalgebra $\kg\subseteq\gg$ with $\kg+\gg^\infty=\gg$ and $\cg=(\cg\cap\kg)+(\cg\cap\gg^\infty)$. 
	If this is the case, then $\cg\cap\kg=\{0\}$ by Step 1, hence $\cg=\cg\cap\gg^\infty$, 
	that is, $\cg\subseteq\gg^\infty$, which completes the proof, by the remarks at the beginning of the proof. 
	
	To prove the above claim, we first note that, since $\gg^\infty$ is abelian, 
	we may use Lemma~\ref{L1} to find a nilpotent subalgebra $\hg\subseteq\gg$ with $\gg=\gg^\infty+\hg$. 
	Denoting $\tg:=\cg+\hg$, we have a Lie algebra isomorphism $\tg/\cg\simeq \hg/(\cg\cap\hg)$. 
	Since $\hg$ is nilpotent, we obtain 
	$\tg^\infty:=\bigcap\limits_{j\ge 1}\tg^j\subseteq\cg$. 
	Since $\tg\subseteq\gg$, we actually have 
	\begin{equation}
		\label{L1_proof_eq1}
		\tg^\infty\subseteq \cg\cap\gg^\infty\quad (\subseteq\cg\subseteq\tg).
	\end{equation} 
	In particular $\tg^\infty$ is abelian hence, by Lemma~\ref{L1} applied for the Lie algebra $\tg$, we obtain a nilpotent subalgebra $\kg\subseteq\tg$ with $$\tg=\tg^\infty+\kg\mathop{=}\limits^{\eqref{L1_proof_eq1}}(\cg\cap\gg^\infty)+\kg.$$ 
	Since $\cg\subseteq\tg$, we then obtain $\cg=(\cg\cap\gg^\infty)+(\cg\cap\kg)$ 
	(which is the second equality in the above claim). 
	Furthermore 
	$$\gg^\infty+\kg=\gg^\infty+(\cg\cap\gg^\infty)+\kg
	=\gg^\infty+\tg\supseteq\gg^\infty+\hg=\gg $$
	which implies $\gg^\infty+\kg=\gg$, and the above claim is completely proved. 
\end{proof}

We recall that a Lie algebra $\gg$ is called metabelian if $[\gg, \gg]$ is abelian.

\begin{proposition}
	\label{P3}
	Let $\gg$ be a finite-dimensional real Lie algebra with trivial centre.
	Then $\gg$ is metabelian if and only if it has an abelian ideal $\ng$ and an abelian subalgebra~$\ag$ satisfying $\gg=\ng\dotplus\ag$. 
	If this is the case, then one may choose $\ng:=\gg^\infty$ 
		and~$\ag$ can be any Cartan subalgebra of~$\gg$.
\end{proposition}

\begin{proof}
	If $\gg=\ng\dotplus\ag$ as in the statement, then $\gg^1:=[\gg,\gg]=[\ng,\ag]\subseteq\ng$, 
	hence $[\gg^1,\gg^1]\subseteq[\ng,\ng]=\{0\}$. 
	
	Conversely, let us assume that $\gg$ is metabelian. 
	Then $\gg^1$ is abelian and, on the other hand, $\gg^\infty\subseteq\gg^1$ and $\gg^\infty$ is a maximal abelian subalgebra of $\gg$ by Lemma~\ref{L2}. 
	Therefore $\gg^\infty=\gg^1$. 
	For every Cartan subalgebra $\hg\subseteq\gg$ we have $\gg=\gg^\infty\dotplus\hg$ by Lemma~\ref{L1}. 
	This implies in particular 
	\begin{equation}
		\label{P3_proof_eq1}
		\gg^1=[\gg,\gg]=[\hg,\hg]+[\hg,\gg^\infty].
	\end{equation}
	Here we have $[\hg,\hg]\subseteq\hg$, $[\hg,\gg^\infty]\subseteq\gg^\infty$, and $\gg^1=\gg^\infty$.
	On the other hand, by the direct sum decomposition  $\gg=\gg^\infty\dotplus\hg$,  we have in particular $\hg\cap\gg^\infty=\{0\}$. 
	Therefore \eqref{P3_proof_eq1} implies $[\hg,\hg]=\{0\}$. 
	Thus the assertion holds with $\ng:=\gg^\infty$ and $\ag:=\hg$. 
\end{proof}

\section{Characterizations of nilpotent Lie groups: proof of Theorem~\ref{IRS}}
\label{proofIRS}

In this section we prove Theorem~\ref{IRS}.
This requires several auxiliary results.

We start by stating two classical results 
to be referred to later on.
We recall from \cite[page 170]{AuMo66} that a 1-connected solvable Lie group $G$ is said to be class~R (or type~R) if for every $g\in G$ its corresponding adjoint action $\Ad_G(g)\colon\gg\to\gg$ is a linear map that has all eigenvalues of modulus one.

\begin{lemma}
	\label{closed}
	If $G$ is a 1-connected solvable Lie group whose coadjoint orbits are simply connected, then the following conditions are equivalent: 
	\begin{enumerate}[{\rm(i)}]
		\item\label{closed_item1} 
		The group $G$ is class~R and type~\I. 
		\item\label{closed_item2}  
		Every coadjoint orbit of $G$ is closed in $\gg^*$. 
		\item\label{closed_item3} 
		 The group $G$ is CCR (liminary). 
	\end{enumerate}
\end{lemma}

\begin{proof}
``\eqref{closed_item2}$\Leftrightarrow$\eqref{closed_item3} 
This is  \cite[Cor., page 48]{Pu78}.

``\eqref{closed_item3}$\Rightarrow$\eqref{closed_item1} 
Since the group $G$ is CCR, it is  type~\I, 
and  class~R by \cite[Ch. V, \S 1, Th. 1, page 171]{AuMo66}. 

``\eqref{closed_item1}$\Rightarrow$\eqref{closed_item3} 
This follows by \cite[Ch. V, \S 1, Th. 1, page 171]{AuMo66} 
even without the hypothesis that the coadjoint orbits of $G$ are simply connected. 
\end{proof}

\begin{lemma}[P.~Hall]
\label{Hall}
Let $\Gamma$ be an arbitrary group. 
Then $\Gamma$ is nilpotent if and only if it has a normal subgroup $N$ 
such that both groups $N$ and $\Gamma/[N,N]$ are nilpotent. 
\end{lemma}

\begin{proof}
See for instance \cite[5.2.10]{Ro96}.
\end{proof}

\begin{remark}
\normalfont 
The following version of Lemma~\ref{Hall} for Lie algebras can be found in \cite[Th. 2]{Ch68}:
If $\gg$ is a finite-dimensional real Lie algebra, then $\gg$ is nilpotent if and only if it has an ideal $\ng\subseteq\gg$ such that both $\ng$ and the quotient $\gg/[\ng,\ng]$ are nilpotent Lie algebras.
\end{remark}

\begin{lemma}
\label{quot_orbits}
Let $H$ be a $1$-connected Lie group and assume that $\ng\subseteq \hg$ is an ideal of the Lie algebra $\hg$ of $H$. 
Denote by $N:=\langle\exp_H(\ng)\rangle$ the analytic subgroup of $H$ corresponding to $\ng$.
Then $N$ is a $1$-connected closed normal subgroup of $H$, and every coadjoint orbit of the quotient group $H/N$ is homeomorphic to a certain coadjoint orbit of $H$. 
Moreover, if the coadjoint orbits of $H$ are closed in $\hg^*$, then the coadjoint orbits of $H/N$ are closed in $(\hg/ \ng)^*$. 
\end{lemma}

\begin{proof}
The subgroup $N$ is $1$-connected and closed in $H$ by \cite[Thm.~11.1.21]{HN12}. 

The coadjoint action of $H$ leaves $\ng^\perp$ invariant:
for all $\xi\in\ng^\perp$ we have $[\hg,\ng]\subseteq \Ker\xi$, hence $\xi\circ\Ad_H(h)\in\ng^\perp$ and  $\xi\circ\Ad_H(n)=\xi$
for every $h\in H$ and $n\in N$. 
Then using the quotient map $q\colon \hg\to\hg/\ng$,
we get the linear isomorphism 
\begin{equation*}
Q\colon (\hg/\ng)^*\to\ng^\perp:=\{\xi\in\hg^* :  \xi\vert_\ng=0\}, 
\quad 
\eta\mapsto \eta\circ q
\end{equation*}
such that
for every $\eta \in (\hg/\ng)^*$ and $h \in H$, 
$$ Q(\Ad_{H/N}^*(p(h)) \eta )= \Ad_H^*(h) (Q\eta), 
$$ 
where $p\colon H\to H/N$, $p(h):=hN$, is the canonical quotient map.
Thus, via the linear isomorphism $Q$, 
every coadjoint orbit of the quotient group $H/N$ is homeomorphic with a certain coadjoint orbit of $H$. 
Moreover, if $\Ad_H^*(H) Q\eta$ is closed in $\gg^*$, it is closed in $\ng^\perp$ as well, hence $\Ad_{H/N}^*(H/N) \eta$ is also closed in $(\hg/\ng)^*$, for every $\eta \in (\hg/\ng)^*$.
\end{proof}

For the following lemma, if $m\ge 1$ is an integer and $\Vc$ is an $m$-dimensional real vector space, we select an arbitrary basis $\underline{v}$ in $\Vc$  and 
	for any $T\in\End_\RR(\Vc)$, we define its spectrum $\spec\, T:=\{z\in\CC : \det(T_{\underline{v}}-z I_m)=0\}$, 
	where $T_{\underline{v}}\in M_m(\RR)$ is the matrix of $T$ with respect to the basis~$\underline{v}$, while $I_m\in M_m(\RR)$ is the identity matrix. 
It is well known and easy to check that the set $\spec\, T$ does not depend on the basis~$\underline{v}$.

\begin{lemma}
	\label{two1}
Let $\Xc$ be a finite-dimensional real vector space regarded as a 1-con\-nected abelian Lie group, and let
$\rho\colon \Xc\to\End_{\RR}(\Vc)$ be a continuous representation on a finite-dimensional real vector space~$\Vc$. 
Consider the corresponding semidirect product $G:=\Vc\rtimes_\rho \Xc$. 
If $\spec\,\rho(x)\subseteq\TT$ for every $x\in \Xc$, the the following assertions are equivalent: 
\begin{enumerate}[{\rm(i)}]
	\item\label{two1_item1} The Lie group $G$ is type~\I\   
and for every $\xi\in \Vc^*$ its orbit $\rho(\Xc)^*\xi\subseteq\Vc^*$ is a simply connected subset of $\Vc^*$. 
\item\label{two1_item2} For every $x\in \Xc$ we have  $\spec\,\rho(x)=\{1\}$. 
\item\label{two1_item3}   The group $G$ is nilpotent. 
\end{enumerate}
\end{lemma}

\begin{proof}
``$\eqref{two1_item2}\implies\eqref{two1_item1}$'' 
This is clear. 

``$\eqref{two1_item1}\implies\eqref{two1_item2}$'' 
We recall from \cite{BB21} that, since $\Xc$ is in particular a nilpotent Lie group, 
we have $\Vc=\Vc_1\dotplus\cdots\dotplus\Vc_m$, a direct sum decomposition into $\rho(\Xc)$-invariant subspaces satisfying the condition that for every $j=1,\dots,m$ there exists a field  $\KK_j\in\{\RR,\CC\}$ and a $\KK_j$-vector space structure of $\Vc_j$ that agrees with its structure of a real vector subspace of~$\Vc$ and there exists a function $\chi_j\colon \Xc\to\KK_j^\times$ such that 
\begin{equation}\label{two1_proof_eq0.5}
\rho(x)\vert_{\Vc_j}\in\End_{\KK_j}(\Vc_j)\; \text{ and }\; (\rho(x)\vert_{\Vc_j}-\chi_j(x)\id_{\Vc_j})^{m_j}=0
\end{equation} 
for all $x\in \Xc$, where $m_j:=\dim_{\KK_j}(\Vc_j)$. 
For every $x\in\Xc$ we have $\spec\,\rho(x)=\{\chi_j(x) :  j=1,\dots,m\}$ hence 
the hypothesis $\spec\,\rho(x)\subseteq\TT$ implies that there exist and are uniquely determined the linear functionals 
$\beta_j\colon\Xc\to\RR$ satisfying
\begin{equation}
	\label{two1_proof_eq1}
\chi_j(x)=\ee^{\ie\beta_j(x)}
\text{ for }j=1,\dots,m.
\end{equation}
for all $x\in\Xc$.
After a relabelling of $\Vc_j$, we may assume that there exists $r\in\{1,\dots,m\}$, 
uniquely determined by the condition $\beta_j=0$ for $j=1,\dots,r$ and $\beta_j\ne 0$ if $r<j\le m$. 

We now prove by contradiction that $r=m$, that is, $\beta_j=0$ for every $j=1,\dots,m$. 
We assume $r<m$ and,  
for $j=r+1,\dots,m$, we apply Sophus Lie's classical theorem for the representation $\rho(\cdot)\vert_{\Vc_j}$ of the abelian Lie group $\Xc$ in the $\CC$-vector space~$\Vc_j$.  
We thus select a $\CC$-vector subspace $\Vc_j^0\subseteq\Vc_j$ that is invariant to $\rho(\cdot)\vert_{\Vc_j}$ and satisfies $\dim_{\CC}(\Vc_j/\Vc_j^0)=1$. 
If we fix $v_j\in\Vc_j\setminus\Vc_j^0$, then for every $x\in\Xc$ the property 
$(\rho(x)\vert_{\Vc_j}-\chi_j(x)\id_{\Vc_j})^{m_j}=0$ implies 
$\spec(\rho(x)\vert_{\Vc_j})=\{\chi_j(x)\}$ hence 
\begin{equation}
\label{two1_proof_eq2}
\rho(x)v_j-\chi_j(x)v_j\in \Vc_j^0. 
\end{equation} 
We now define 
$$\Vc^0:=\bigoplus_{j=1}^r\Vc_j\oplus\bigoplus_{j=r+1}^m\Vc_j^0.$$
Then $\Vc^0$ is a $\rho(G)$-invariant $\RR$-vector subspace of $\Vc$, 
hence a normal subgroup of $\Vc\rtimes\Xc=G$. 
Since the group $G$ is type~\I, it follows that the quotient group $G/\Vc^0$ is type \I. 
Denoting $\Vc_0:=\Vc/\Vc^0$, there is a representation $\rho_0\colon \Xc\to\End_\CC(\Vc_0)$ for which we have a Lie group isomorphism $G/\Vc^0\simeq \Vc_0\rtimes \Xc=:G_0$. 
Then, the hypothesis~\eqref{two1_item1} implies that for every $\eta\in\Vc_0^*$ 
its orbit $\rho_0^*(\Xc)\eta\subseteq\Vc_0^*$ is a simply connected subset of~$\Vc_0^*$. 
(Compare the proof of Lemma~\ref{quot_orbits}.)
Moreover, we note that $\Vc_0$ is a $\CC$-vector space with the basis $\{v_{0j}:=v_j+\Vc^0\}_{j=r+1,\dots,m}$, and we have 
\begin{equation}
\label{two1_proof_eq3}
 \rho_0(x)v_{0j}=\chi_j(x)v_{0j}=\ee^{\ie\beta_j(x)}v_{0j} 
\text{ for }j=r+1,\dots,m
\end{equation}
for all $x\in\Xc$, 
by \eqref{two1_proof_eq1}--\eqref{two1_proof_eq2}. 
One now obtains a contradiction.   
Specifically,  
it follows by \cite[Eq. (4.8)]{BB21} that, for $(\eta,0)\in \Vc_0^*\times\Xc^*\subseteq \Vc^*\times\Xc^*=\gg^*$, its coadjoint orbit $G.(\eta,0)\subseteq\gg^*$ is homotopy equivalent to the orbit $\rho_0^*(\Xc)\eta\subseteq\Vc_0^*$, which is simply connected, as we have proved above. 
On the other hand, the fundamental group of the coadjoint orbit $G.(\eta,0)\subseteq\gg^*$ 
is nontrivial by \cite[Lemma 7.3]{ACD12} since~$r<m$.

Alternatively, one can argue as follows. 
We endow $\Vc_0$ with a complex scalar product $(\cdot\mid\cdot)$ for which the set $\{v_{0j}\}_{j=r+1,\dots,m}$ is an orthonormal basis, hence $\rho_0$ is a representation of $\Xc$ by unitary diagonal operators on the finite-dimensional complex Hilbert space $\Vc_0$.

On the other hand, the condition  $\spec\,\rho_0(x)\subseteq\TT$ for every $x\in\Xc$ implies that 
the group $G_0$ is class~R. 
Then, since $G_0$ is type~\I\ by hypothesis, it follows by \cite[Ch. V, \S 1, Th. 2]{AuMo66} 
that $G_0$ is CCR. 
Since $C^*(G_0)\simeq\Cc_0(\Vc_0^*)\rtimes \Xc$ and $\Xc$ is abelian, it then follows by \cite[Th. 8.44]{Wi07} that for every $\xi\in \Vc_0$ its orbit $\rho_0^*(\Xc)\xi\subseteq\Vc_0^*$ is a closed subset of $\Vc_0^*$.
Using the $\Xc$-equivariant isomorphism of finite-dimensional $\RR$-vector spaces $\Vc_0\to\Vc_0^*$, $v\mapsto\Re(\cdot\mid v)$, 
it then follows that for every $v\in\Vc_0$ its orbit $\rho_0(\Xc)v\subseteq\Vc_0$ is a closed subset of $\Vc_0$. 
Using again the fact that $\rho_0(\Xc)$ is a group of unitary operators on the finite-dimensional Hilbert space~$\Vc_0$, it then follows that  for every $v\in\Vc_0$ its orbit $\rho_0(\Xc)v\subseteq\Vc_0$ is actually a compact subset of $\Vc_0$. 

Moreover, for $v_0:=\sum\limits_{j=r+1}^mv_{0j}$ it follows by \eqref{two1_proof_eq3} that the mapping 
$$\rho_0(\Xc)\to \rho_0(\Xc)v_0,\quad \rho_0(x)\mapsto\rho_0(x)v_0$$
is a homeomorphism, hence the connected abelian group of unitary operators $\rho_0(\Xc)$ is compact, and then it is homeomorphic to a torus. 
Thus the orbit $\rho_0(\Xc)v_0\subseteq\Vc_0$ is homeomorphic to a torus, 
and then it is not simply connected, which is a contradiction with our hypothesis. 

``$\eqref{two1_item2}\iff \eqref{two1_item3}$'' 
For arbitrary $x\in\Xc$ and $t\in\RR$ we have $\rho(tx)=\ee^{\de\rho(tx)}=\ee^{t\de\rho(x)}$. 
This shows that \eqref{two1_item2} is equivalent to $\spec\,\de\rho(x)=\{0\}$ for every $x\in\Xc$, which is further equivalent to the fact that for every $x\in \Xc$ the endomorphism $\de\rho(x)\in\End_\RR(\Vc)$ 
is nilpotent, that is, $\de\rho(x)^m=0$, where $m:=\dim_\RR\Vc$. 
This last condition is clearly equivalent to the fact that the semidirect product of abelian Lie algebras $\gg:=\Vc\rtimes_{\de\rho}\Vc$ is nilpotent, 
which is further equivalent to~\eqref{two1_item3}.
\end{proof}

\begin{lemma}
\label{two}
Theorem~\ref{IRS} holds true if the group $G$ is the semidirect product of two 1-connected abelian Lie groups. 
\end{lemma}

\begin{proof}
We may assume that $G=\Vc\rtimes \Kc$, where both $\Vc$ and $\Kc$ are vector groups. 
Since $\Kc$ is abelian, it follows by 
\cite[Prop. 4.4(1.(b))]{Ba98}  (and see also \cite[Rem.~4.15]{BB21})
that every coadjoint orbit of $G$ is diffeomorphic to the cotangent bundle of some orbit of $\Kc$ in the dual space $\Vc^*$ 
and conversely, every cotangent bundle of a $\Kc$-orbit in $\Vc^*$ occurs in this way from some coadjoint orbit of $G$. 
Thus, the hypothesis that the coadjoint orbits of $G$ are simply connected and closed is equivalent to the condition that all the $\Kc$-orbits in $\Vc^*$ are simply connected and closed.
The statement is then a consequence of Lemma~\ref{two1}.
\end{proof}

\begin{proof}[Proof of Theorem~\ref{IRS}]
``$\Rightarrow$'' If $G$ is nilpotent, its coadjoint orbits are closed and 
homeomorphic to $\RR^k$ for some  integer $k\ge 0$ (\cite[Thm.~3.1.4]{CG90}). 

``$\Leftarrow$''
We proceed by
induction on~$\dim G$. 
The case $\dim G=1$ is clear. 
We now assume that the statement holds true for all solvable Lie groups having their dimension less than $\dim G$.

There are two possible cases, depending on the commutator subgroup $D:=[G,G]$, which is a closed 1-connected nilpotent subgroup of $G$. 

Case 1: $\dim[D,D]>0$. 
We consider the quotient Lie group $H:=G/[D,D]$, with Lie algebra $\hg$.
The coadjoint orbits of $H$ are simply connected  and closed in $\hg^*$ by the hypothesis on~$G$ along with Lemma~\ref{quot_orbits}. 
Since $\dim[D,D]>0$, we have $\dim H <\dim G$ hence, by the induction hypothesis, 
it follows that $H$ is a nilpotent Lie group. 
On the other hand $[D, D]$ is a nilpotent group as well, hence Lemma~\ref{Hall} shows that the group $G$ is nilpotent. 

Case 2: $[D,D]=\{\1\}$. 
Since $D=[G,G]$, it follows that the group $G$ is metabelian, 
and there are two possible subcases, depending on the center $Z$ of~$G$. 
\begin{itemize}
	\item[(2a)] $\dim Z>0$. 
	The quotient group $G/Z$ has simply connected coadjoint orbits that are closed in $\gg^*/\zg^*$, 
	by the hypothesis on~$G$ along with Lemma~\ref{quot_orbits}. 
	Since $\dim G/Z<\dim G$, it follows by the induction hypothesis that $G/Z$ is a nilpotent group. 
	Since $Z$ is the centre of $G$, this implies that the group $G$ is nilpotent. 
	\item[(2b)] $Z=\{\1\}$. 
	In this subcase, Proposition~\ref{P3} is applicable and shows that $G$ is the semidirect product of two 1-connected abelian Lie groups. 
	Then the assertion follows by Lemma~\ref{two}.
	\end{itemize}
This completes the proof by induction on $\dim G$. 
\end{proof}

\begin{remark}\label{Boidol}
\normalfont
We note that Theorem~\ref{IRS} is no longer valid if we assume that the coadjoint orbits are closed  and simply connected only for 
an open dense set of coadjoint orbits. 
Indeed, consider the exponential Lie group $G$, whose Lie algebra~$\gg$ has a basis 
$\{X_1, X_2, X_3, X_4\}$
and non-trivial brackets
$$ [X_4, X_3]=-X_3, \; [X_4, X_2]= X_2, \; [X_3, X_2]=X_1.$$
The group $G$ is called the split oscillator, or the Boidol group. 
Since $G$ is exponential, all its coadjoint orbits are simply connected. 
If we denote by $X_1^*, X_2^*, X_3^*, X_4^*$ the dual basis in $\gg^*$, the coadjoint orbits $\Ad^*(G)\ell$ with 
$$\ell=\xi_4 X_4^* + \xi_1 X_1^*, \quad \xi_4\in \RR,\; \xi_1\in \RR^\times $$ 
are closed and form an open dense subset of $\gg^*/G$. 
Nevertheless, the orbits $\Ad^*(G) \ell$ with $\ell= \xi_2 X_2^* +\xi_3 X_3^*$, 
$\xi_2^2+\xi_3 ^2\ne 0$ are not closed, and the group $G$ is not nilpotent.
See \cite[Subsect.~2.2]{LLu24} for more details.
\end{remark}

\section{Proof of Theorem~\ref{main-main_thm}}
\label{proofMain}

Before starting the proof of Theorem~\ref{main-main_thm}, 
we recall some facts on the Heisenberg group $H_n$ and its $C^*$-algebra. 

Fix $n\in \NN$.  
The Lie algebra $\hg_n$  of $H_n$ is generated by 
a basis $\{X_0, X_1, \cdots, X_{2n}\}$
with the non-trivial Lie brackets given by
$$ [X_j, X_{j+n}] = X_0, \qquad j=1, \dots, n.$$
The unitary dual of $H_{n}$ is 
$$ \widehat{H_{n}}=\widehat{C^*(H_n)} =\Gamma_\infty\sqcup\Gamma_1
$$ 
The points in $\Gamma_\infty\simeq\RR^\times$ correspond to the infinite dimensional representations, unitary equivalent to the Schr\"odinger representations, 
while $\Gamma_1\simeq\RR^{2n}$ consists of characters (1-dimensional characters). 
The Hausdorff open subset $\Gamma_\infty$ is dense in~$\widehat{H_{n}}$, and moreover, a
subset $F\subseteq\widehat{H_{n}}$ is closed if and only if $F\cap \Gamma_j$ is closed in~$\Gamma_j$ for $j=1, \infty$, and 
 if $0\in\RR$ is an accumulation point of $F\cap\Gamma_\infty$, then $\Gamma_1\subseteq F$. 

A description of $C^*(H_n)$ as an algebra of operator fields can be found in \cite{LuT11}.

\subsection{Puk\'anszky's correspondence and coadjoint orbits}
We need some results on Puk\'anszky's correspondence in the special case of a locally closed coadjoint orbit, 
so let us give a sketchy reminder of the construction in this case.

Assume first that $G$ is a 1-connected solvable Lie group
with its Lie algebra $\gg$.
The Lie algebra of the closed connected normal subgroup 
\begin{equation}\label{D_def}
D:=[G,G]
\end{equation}  of $G$ is 
the derived ideal $\dg:=[\gg,\gg]$. 
For $\xi\in \gg^*$  arbitrary fixed, we denote by
$G(\xi):=\{g\in G :  g\xi=\xi\}$  the stabilizer at $\xi$ 
 and let
	 $G(\xi)_\1$ denote the connected component of $\1\in G(\xi)$. 
Here $G(\xi)$ is a closed subgroup of $G$ with its Lie algebra denoted by $\gg(\xi)$, while $G(\xi)_\1$ is an 1-connected closed subgroup of $G$ whose Lie algebra is again~$\gg(\xi)$.
 
Now let $\Oc \in \gg^*/G$ be fixed and assume that $\Oc$ is a locally closed subset of $\gg^*$. 
For $\xi\in \Oc$, the group $K:=D G(\xi):=\{xg : x\in D,g\in G(\xi)\}$ is a closed subgroup of~$G$ independent on $\xi \in \Oc$, and the connected component of $\1\in K$ is 
$K_\1= DG(\xi)_\1$. 
Then the quotient $\Pi= K/K_\1\simeq G(\xi)/G(\xi)_\1$ is a finitely generated free abelian group, isomorphic with $\ZZ^{\rk(\Oc)}$, where $\rk(\Oc)$ is a non-negative integer.

 Denote $$\mathop{G}\limits^{\mathsout{\wedge}}(\xi):=
	\{\chi\in\Hom(G(\xi),\TT) :  \chi\vert_{G(\xi)_\1}=\chi_\xi \}, $$
	where $\chi_\xi(\exp X) =\ee^{\ie\langle \xi, X\rangle}$ for $X\in \gg(\xi)$.
	Consider then 
$$\Bun(\Oc):=\bigsqcup_{\xi\in\Oc}\{\xi\}\times \mathop{G}\limits^{\mathsout{\wedge}}(\xi). 
$$
The group $G$ acts naturally on $\Bun (\Oc)$.
Also, the group
\begin{equation}\label{pi1} 
 \widetilde{\Pi}= \{ \phi\in \Hom(K, \TT) :  
\phi\vert_{K_1}=1\}\simeq \Hom(\Pi, \TT)\simeq \TT^{\rk(\Oc)},
\end{equation}
acts on $\Bun(\Oc)$ by 
$$ p=(\xi, \chi) \mapsto p \chi = (\xi, \chi \phi\vert_{G(\xi)})$$
for every $\phi\in \widetilde{\Pi}$, and the actions of $G$ and 
$\widetilde{\Pi}$ commute.
This gives 
 a natural $G$-equivariant bijection 
$\Bun(\Oc)\to \Oc\times\TT^{\rk(\Oc)}$ 
and $\Bun(\Oc)$ is endowed with the smooth manifold structure transported from 
$\Oc\times\TT^{\rk(\Oc)}$ via the above bijection.
(See \cite[Subsect. 6.3, page 537]{Pu71}.)
Moreover
$$\tau\colon \Bun(\Oc)\to \Oc, \quad \tau(\xi, \chi)= \xi$$
is a principal bundle with structural  group $\widetilde{\Pi}$.

\begin{remark}\label{Pi-remark}
\normalfont
(a) Let $G$ and $\Oc$ be as above.
The group $\widetilde{\Pi}$ is isomorphic to the dual $\widehat{\pi_1(\Oc)}$ of the fundamental group 
$\pi_1(\Oc)$ of $\Oc$.  
Indeed, for every $\xi \in \Oc$, $\Pi\simeq G(\xi)/G(\xi)_\1\simeq \pi_1(\Oc)$ (see \cite[page 258]{AuKo71}). 

(b) 
If we assume that $G$ is type I, then Puk\'anszky's correspondence
is a bijection 
$$\kappa\colon  \bigsqcup_{\Oc \in \gg^*/G} \Bun(\Oc)/G \to \widehat{G}.$$
(See also \cite[Thm. V.3.3, page 259]{AuKo71} and Lemma~\ref{trivial} below.)
Moreover, the restriction $$\kappa\vert_{\Bun(\Oc)/G}\colon 
\Bun(\Oc)/G \to \widehat{G}$$ is continuous, by 
\cite[Lemma 8, p. 93 and Thm. 1, p. 114]{ Pu73}.
\end{remark}

 \begin{lemma}\label{trivial}
If $\Oc \in \gg^*/G$ be locally closed, 
then $\Bun(\Oc)/G$ is homeomorphic to~$\widetilde{\Pi}$. 
\end{lemma}
 
 \begin{proof}
 There exists a smooth, $G$-equivariant 
 mapping $\sigma\colon \Oc \to \Bun(\Oc)$ such that $\tau \circ \sigma=\id_\Oc$, by \cite[Proof of Prop.~5.2]{BB21}.
 Then, since $\Bun(\Oc)$ is a principal bundle with structural group $\widetilde{\Pi}$, the map defined by
 $$ \Psi\colon \Oc \times
 \widetilde{\Pi} \to \Bun (\Oc), 
 \quad (\xi, z)\mapsto \sigma(\xi) z $$
 is a diffeomorphism.
 Since $\sigma$ is equivariant, we have that 
 $ \Psi(g \xi, z) = g \Psi(\xi, z)$ for every $g\in G$, $\xi \in \Oc$ and 
 $z\in \widetilde{\Pi}$.
Therefore, for arbitrary $\xi_0\in\Oc$, the map 
 $$ \Psi_{\xi_0} \colon \widetilde{\Pi}\to \Bc(\Oc)/G, \quad  
 z\mapsto G \Psi(\xi_0, z)= G\sigma(\xi_0)z = \sigma(G\xi_0)z =\sigma(\Oc) z$$
 is correctly defined. 
 Here $G\xi_0=\Oc$ since $\xi_0\in\Oc$, while $\Oc\in\gg^*/G$ by hypothesis.
 Moreover, the diagram 
 $$ 
\xymatrix{ 
\Oc \times \widetilde{\Pi}
 \ar[d]_{\pr_2} \ar[rr]^{\Psi} &  &
\Bun(\Oc) \ar[d]^{q}\\ 
\widetilde{\Pi}
 \ar[rr]^{\Psi_{\xi_0}}  & & \Bun(\Oc)/G
}
$$
is commutative, and since $\Psi$ is a homeomorphism, it follows that 
$\Psi_{\xi_0}$ is a homeomorphism. 
 \end{proof}

\subsection{$C^*$-rigidity of the Heisenberg group within the class of solvable Lie groups}

This subsection is devoted to the proof of the following result.

\begin{proposition}\label{main_thm}
Let $G$ be a solvable 1-connected Lie group such that
 $C^*(G) \simeq C^*(H_{n})$.
 Then 
 $G\simeq H_{n}$.
\end{proposition}

Assume from now on that $G$ is a solvable 1-connected Lie group such that 
 $C^*(G) \simeq C^*(H_{n})$.  
Hence the group $G$ is CCR, in particular type~$\I$, and
 \begin{equation*}
 \widehat{G} = S_1 \sqcup S_\infty, 
\end{equation*}
where $S_1=\Hom(G, \TT)$ is the set of characters, while
\begin{equation*}
S_\infty \simeq \RR^{\times}
\end{equation*} 
 corresponds to infinite-dimensional irreducible representations.

 The group $G$ is CCR, hence for every $\xi \in \gg^*$, the coadjoint orbit 
 $\Oc_\xi=\Ad^*(G)\xi$ is closed in $\gg^*$, by Lemma~\ref{closed}.

 \begin{lemma}\label{4}
For every $\xi \in \gg^*$, the stabilizer $G(\xi)$ is connected and 
the coadjoint orbit $\Oc_\xi$ is simply connected. 
\end{lemma}
 
 \begin{proof}
 There is a bijective mapping
 $$ \kappa\colon \bigsqcup\limits_{\Oc \in \gg^*/G} 
 \Bun (\Oc)/G \to \widehat{G}, $$
 and the restriction 
 $ \kappa\vert_{\Bun(\Oc_\xi) /G}\colon 
 {\Bun(\Oc_\xi)/G} \to \widehat{G}$
 is continuous. (See Remark~\ref{Pi-remark}.)
 On the other hand, $\Bun(\Oc_\xi)/G$ is a torus, by Lemma~\ref{trivial}, , and 
 $\kappa(\Bun(\Oc_\xi)/G)\subseteq S_\infty\simeq\RR^\times$,
 for every $\xi \in \gg^*\setminus [\gg, \gg]^\perp$.
 Hence
 $\kappa(\Bun(\Oc_\xi)/G)$ must be a point in $S_\infty$. 
 It follows that $\widehat{\pi_1(\Oc_\xi)} =\{1\}$, therefore 
 $\pi_1(\Oc_\xi)=\{\0\}$. 
 We have thus obtained that $\Oc_\xi$ is simply connected, hence $G(\xi)$ is connected 
 for $\xi \in \gg^*\setminus [\gg, \gg]^\perp$. 
If $\xi \in [\gg, \gg]^\perp$ then $G(\xi)=G$. 
Thus, for arbitrary $\xi\in\gg^*$, the stabilizer $G(\xi)$ is connected. 
Then, using the diffeomorphism defined by the coadjoint action, 
$G/G(\xi)\to\Oc_\xi$, $gG(\xi)\mapsto g\xi$, and the fact that the Lie group $G$ is simply connected, 
it follows that the coadjoint orbit $\Oc_\xi$ is simply connected, too. 
 \end{proof}

\begin{proof}[Proof of Proposition~\ref{main_thm}]
By Lemma~\ref{4} and the discussion before it, the coadjoint orbits of the group $G$ are simply connected and closed in $\gg^*$. 
It follows that $G$ is nilpotent, by Theorem~\ref{IRS}. 
Now the statement is a consequence of the fact that the Heisenberg group is $C^*$-rigid within the class of exponential Lie groups, cf. \cite[Thm.~1.2]{BB-JTA}.
\end{proof}

\subsection{Beyond solvable Lie groups} 

\begin{lemma}
	\label{stack1}
	If $X$ is a connected topological space and there exist topological spaces $X_1$ and $X_2$, 
	containing each of them at least two points, satisfying $X= X_1\times X_2$, 
	then the complement of every singleton subset of $X$ is connected. 
\end{lemma}

\begin{proof}
	For arbitrary $x=(x_1,x_2)\in X_1\times X_2=X$ we must show that $X\setminus\{x\}$ is connected. 
	To this end we use the hypothesis that $X_1$ and $X_2$, 
	contains at least two points, hence there exist $y_j\in X_j\setminus\{x_j\}$ for $j=1,2$. 
	Then $y:=(y_1,y_2)\in X\setminus\{x\}$,  
	and we claim that for arbitrary $a=(a_1,a_2)\in X\setminus\{x\}$ there exists a connected subset $C\subseteq X\setminus\{x\}$ with $a,y\in C$. 
	
	We first note that the Cartesian projection map of $X=X_1\times X_2$ onto~$X_j$ 
	is continuous for $j=1,2$ and, since the image of a connected space through a continuous map is connected, it follows that both $X_1$ and $X_2$ are connected topological spaces. 
	Since $(a_1,a_2)=a\ne x=(x_1,x_2)$, there are two possible cases. 
	
	Case $a_1\ne x_1$. 
	We then define $C:=(\{a_1\}\times X_2)\cup(X_1\times\{y_2\})\subseteq X$. 
	Since $a_1\ne x_1$ and $y_2\ne x_2$, we have $(x_1,x_2)\not\in C$, hence $C\subseteq X\setminus\{x\}$. 
	Both subsets $\{a_1\}\times X_2$ and $X_1\times\{y_2\}$ are connected, since they are homeomorphic to $X_2$ and $X_1$, respectively. 
	Moreover, $(a_1,y_2)\in (\{a_1\}\times X_2)\cap(X_1\times\{y_2\})$, 
	hence $(\{a_1\}\times X_2)\cap(X_1\times\{y_2\})\ne\emptyset$. 
	In every topological space, the union of two connected subsets having nonempty intersection is again connected, hence $C$ is indeed a connected subset of $X\setminus\{x\}$. 
	Moreover $a,y\in C$, as needed. 
	
	Case $a_2\ne x_2$. 
	We now define $C:=(X_1\times \{a_2\})\cup(\{y_1\}\times X_2)$ and then  
	$C\subseteq X\setminus\{x\}$ since $x_2\ne a_2$ and $x_1\ne y_1$. 
	Moreover, $(y_1,a_2)\in (X_1\times \{a_2\})\cap(\{y_1\}\times X_2)$, 
	and this implies as above that $C$ is connected. 
	In addition, $a,y\in C$. 
	
	This completes the proof of the fact that that for every $a=(a_1,a_2)\in X\setminus\{x\}$ there exists a connected subset $C\subseteq X\setminus\{x\}$ with $a,y\in C$, 
	and which further implies that $X\setminus\{x\}$ is connected. 
\end{proof}

\begin{lemma}
	\label{stack2}
	If $X_1$ and $X_2$ are topological spaces satisfying $\widehat{H_{n}}=X_1\times X_2$, then one of the spaces $X_1$ and $X_2$ is a singleton. 
\end{lemma}

\begin{proof}
	It follows by the description of $\widehat{H_{n}}=\Gamma_\infty\sqcup\Gamma_1$ at the beginning of the present section that for every $\gamma\in\Gamma_\infty$ the space  $\widehat{H_{n}}\setminus\{\gamma\}$ is not connected.
	Thus, the assertion follows by Lemma~\ref{stack1}. 
\end{proof}

\begin{proposition}\label{solv}
	If $G$ is a 1-connected Lie group with $C^*(G)\simeq C^*(H_{n})$ for some $n\ge 1$, 
	then $G$ is a solvable Lie group. 
\end{proposition}

\begin{proof}
	Since $C^*(G)\simeq C^*(H_{n})$, it follows that 
	$\Prim(G)$ is a $T_1$ topological space. 
	Since the Lie group $G$ is 1-connected, 
	it then follows that $G=G_1\times G_2$, where $G_1$ is a 1-connected semisimple Lie group 
	and $G_2$ is a 1-connected Lie group whose radical is cocompact and class~$R$. 
	(See e.g., \cite[Th. 1]{Pe83} and the references therein.)
	Using the Levi-Malcev decomposition and denoting by $S$ the radical of~$G_2$, 
	we obtain a semidirect product decomposition $G_2=S\rtimes K$, where 
	$K$ is a 1-connected semisimple Lie group. 
	(See \cite[Cor. 5.6.8 and Prop. 11.1.19]{HN12}.)
	Summarizing, 
	$$G=G_1\times (S\rtimes K)$$ 
	where $G_1$ is a 1-connected semisimple Lie group, $K$ is a 1-connected compact semisimple Lie group, and $S$ is a 1-connected solvable Lie group of class~R. 
	Here $G_1\times S$ is a closed normal subgroup of $G$ with $G/(G_1\times S)=K$, 
	hence every irreducible unitary representation $\pi\colon K\to\Bc(\Hc_\pi)$ lifts to an irreducible unitary representation of $G$ on $\Hc_\pi$, which further integrates to an irreducible $*$-representation of $C^*(G)\simeq C^*(H_{n})$ on $\Hc_\pi$. 
	If $K\ne\{\1\}$ then, since $K$ is a 1-connected compact Lie group, it follows that $K$ is actually a compact semisimple (hence noncommutative) Lie group, hence its unitary irreducible representation $\pi$ satisfies $1<\dim\Hc_\pi<\infty$. 
	But then one cannot have an unitary irreducible representation of the Heisenberg group $H_{n}$ on $\Hc_\pi$. 
	Therefore $K=\{\1\}$, and we obtain 
$$G=G_1\times S.$$
Here  $G_1$ is a 1-connected semisimple Lie group, hence CCR, 
and then its full $C^*$-algebra $C^*(G_1)$ is nuclear. 
(See e.g., \cite[Cor. B.49]{RaWi98}.) 
We further obtain $C^*(G)=C^*(G_1)\widehat{\otimes} C^*(S)$ by e.g., 
\cite[Prop. 3.11 and Lemma 2.73]{Wi07}. 
Here $\widehat{\otimes}$ stands for the unique $C^*$-tensor product, 
	i.e., the completion of the algebraic tensor product with respect to its unique $C^*$-norm, which makes sense since 
 $C^*(G_1)$ is nuclear, cf. \cite[page 258]{RaWi98}. 
Moreover, 
both $C^*(G_1)$ and $C^*(S)$ are separable, 
we may use \cite[Th. B.45(c)]{RaWi98} to obtain a homeomorphism $\Prim(G_1)\times\Prim(S)\simeq \Prim(G)$. 

As $\Prim(G)$ is homeomorphic to $\Prim(H_{n})$ by hypothesis, 
it then follows by Lemma~\ref{stack2} that either $\Prim(S)$ or $\Prim(G_1)$ is a singleton. 
Since $S$ and $G$ are connected groups, this implies by \cite[Prop. 4]{BCdH94} that either $G_1=\{\1\}$ or $S=\{\1\}$. 

If  $S=\{\1\}$, then $G=G_1$ is semisimple, 
hence the only 1-dimensional irreducible unitary representation of $G$ is the trivial one, 
and this is also the case for $C^*(G)$. 
On the other hand, by hypothesis $C^*(G)\simeq C^*(H_{n})$, 
and the Heisenberg group~$H_{n}$ (hence also its $C^*$-algebra) admits nontrivial 1-dimensional irreducible unitary representations. 
This contradiction shows that we actually have $G_1=\{\1\}$, hence $G=S$ is solvable, and this completes the proof. 
\end{proof}

\begin{proof}[Proof of Theorem~\ref{main-main_thm}] 
The result is a consequence of Propositions~\ref{solv} and ~\ref{main_thm}.
\end{proof}

\subsection*{Acknowledgment}
We wish thank the Referee for carefully reading our manuscript 
and for many pertinent remarks that have improved the exposition.

\end{document}